\documentclass[12pt]{amsart}
\usepackage{xypic,amssymb,url,tikz}
\usepackage[margin=1.5in]{geometry}

\usepackage[colorlinks=true]{hyperref}
\hypersetup{
    pdftitle={Critical groups of graphs and involutions},
    pdfauthor={Andrew Berget},     % author
    pdfsubject={Critical groups of graphs and involutions},
    pdfcreator={Andrew Berget},   % creator of the document
    pdfproducer={pdfLatex}, % producer of the document
}

\xyoption{all}
\usetikzlibrary{arrows}

\def\QQ{\mathbf{Q}}
\def\ZZ{\mathbf{Z}}
\def\<{\langle}
\def\>{\rangle}
\def\coker{\operatorname{coker}}
\def\Hom{\operatorname{Hom}}
\def\im{\operatorname{im}}
\def\Adj{{\mathbf{Adj}}}
\def\bb{\operatorname{b}}

\newtheorem{theorem}{Theorem}[section]
\newtheorem*{theorem*}{Main Theorem}
\newtheorem{proposition}[theorem]{Proposition}
\newtheorem{corollary}[theorem]{Corollary}
\newtheorem{lemma}[theorem]{Lemma}
\newtheorem{example}[theorem]{Example}

\newtheorem*{convention*}{Convention}
\newtheorem*{claim*}{Claim}

\author{Andrew Berget} \title{Critical groups of graphs with
  reflective symmetry}
\email{aberget@uw.edu}
\begin{document}

\begin{abstract}
  The critical group of a graph is a finite abelian group whose order
  is the number of spanning forests of the graph. For a graph $G$ with
  a certain reflective symmetry, we generalize a result of
  Ciucu--Yan--Zhang factorizing the spanning tree number of $G$ by
  interpreting this as a result about the critical group of $G$. Our
  result takes the form of an exact sequence, and explicit connections
  to bicycle spaces are made.
\end{abstract}

\address{Department of Mathematics, University of Washington, Seattle}

\thanks{The author was partially supported as a VIGRE Fellow at UC
  Davis by NSF Grant DMS-0636297.}
\maketitle

\section{Introduction and Statement of Results}
A graph with \textbf{reflective symmetry} is a graph $G=(V,E)$ with a
distinguished, non-degenerate drawing in $\mathbf{R}^2$ such that
\begin{enumerate}
\item reflection about a line $\ell$ takes the drawing into itself,
  and
\item every edge that is fixed by this reflection about $\ell$ is
  fixed point-wise.
\end{enumerate}
For a graph with reflective symmetry, the reflection of the
distinguished drawing gives rise to an involution. This involution
will always be denoted $\phi$, and is a map $V \to V$ that induces a
map $E \to E$.

Condition (2) above means that no edge of $G$ crosses the axis of
symmetry.  Assuming that the line of reflection is vertical, the
drawing of $G$ gives rise to a partition of the edges $E$ of $G$ into
three blocks: $E = E_L \cup E^\phi \cup E_R$. The sets $E_L$ and $E_R$
denote the edges on the left and right sides of the reflection line
$\ell$, respectively. The set $E^\phi$ denotes the set of $\phi$-fixed
edges. Similarly, there is a partition of the vertices of $G$ as $V =
V_L \cup V^\phi \cup V_R$.

A subgraph of $G$ is specified by the edges of $G$ it contains, its
vertices being the endpoints of the specified edges.  A graph $G_+ =
(V_+,E_+)$ is obtained from $E_L \cup E^\phi$ by subdividing each
$\phi$-fixed edge (i.e., the edges in $E^\phi$).  A graph
$G_-=(V_-,E_-)$ is obtained from $E_R $ by identifying its
$\phi$-fixed vertices to a single vertex. 

\begin{example}
  Below we have a graph $G$ with reflective symmetry, along with $G_+$
  and $G_-$ (shown left-to-right). The shaded vertex is the one
  obtained by subdividing the $\phi$-fixed edge of $G$.
  \tikzstyle{vertex}=[draw, circle, minimum size=10pt, inner sep=1pt]
  \tikzstyle{subvertex}=[draw, circle, fill=red!40,minimum size=10pt,
  inner sep=1pt] \tikzstyle{edge} = [draw,thick,-]
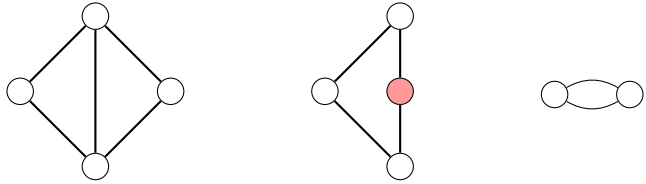
$$
\begin{tikzpicture}
  \foreach \pos/\name in { {(0,-1)/b}, {(0,1)/c},{(-1,0)/a}, {(1,0)/d}}
    \node (\name) at \pos [vertex] {};
  \foreach \source/ \dest in { a/b, a/c, d/c,d/b,c/b}
    \path[edge] (\source) --  (\dest);
\end{tikzpicture}
\qquad\qquad
\begin{tikzpicture}
  \node (bc) at (0,0) [subvertex] {};
    \foreach \pos/\name in { {(0,-1)/b}, {(0,1)/c},
                            {(-1,0)/a}}
        \node (\name) at \pos [vertex] {};
    \foreach \source/ \dest in { a/b, a/c, c/bc,bc/b}
    \path[edge] (\source) --  (\dest);
\end{tikzpicture}
\qquad\qquad
\begin{tikzpicture}
  \node at (0,1) {}; \node at (0,-1) {};
    \foreach \pos/\name in { {(0,0)/bc},
                             {(1,0)/d}}
        \node (\name) at \pos [vertex] {};
    \draw [-] (bc) to [bend left](d);
    \draw [-] (d) to [bend left](bc);
\end{tikzpicture}
$$

\end{example}

This paper will be concerned with the critical group of a graph
possessing reflective symmetry. The critical group of a graph $G$ is a
finite abelian group, denoted $K(G)$, whose order is the number of
spanning forests of $G$. The number of spanning forests of a graph $G$
will be denoted $\kappa(G)$. 

The critical groups is also known as the \textit{Jacobian group} or,
\textit{Picard group} of the graph, and is intimately connected with
the \textit{abelian sandpile model} and a \textit{chip firing game}
played on the vertices of $G$. We will define $K(G)$ formally in
Section~\ref{sec:critical}, and discuss functorial properties in some
depth in the appendix. We will not discuss any connections to
chip firing games.

A theorem of Ciucu, Yan and Zhang \cite{ciucuYanZhang} motivates our
work.
\begin{theorem}[Ciucu--Yan--Zhang]\label{thm:motivation}
  Let $G$ be a planar graph with reflective symmetry, with axis of symmetry $\ell$. Then
  \[
  \kappa(G) = 2^{\omega(G)} \kappa(G_+) \kappa(G_-),
  \]
  where $\omega(G)$ is the number of bounded regions intersected by
  $\ell$.
\end{theorem}

We offer a generalization of their result at the level of critical
groups.
\begin{theorem*}
  Let $G$ be a graph with reflective symmetry, and $G_+$, $G_-$ the
  two graphs this symmetry gives rise to. There is a group
  homomorphism $$f^*:K(G_+) \oplus K(G_-) \to K(G),$$ such that $\ker(f^*)$
  and $\coker(f^*)$ are all $2$-torsion.
  
  The kernel and cokernel of $f$ can be explicitly determined as the
  space of bicycles in $G$ and $G_+ \cup G_-$, respectively, possessing
  certain symmetry properties. If $G_+$ is connected then,
  \[
  \frac{|K(G)|}{|K(G_+) \oplus K(G_-)|}=\frac{|\coker(f^*)|}{|\ker(f^*)|} = 2^{|V^\phi| - |E^\phi| - 1}.
  \]
\end{theorem*}
One obtains a version of Theorem~\ref{thm:motivation} by taking the
orders of the groups in the exact sequence
\[
0 \to \ker(f^*) \to K(G_+) \oplus K(G_-) \to K(G) \to \coker(f^*) \to
0.
\]
\begin{corollary}
  If $G$ is a graph with reflective symmetry and $G_+$ is connected
  then,
  \[
  \kappa(G) = 2^{|V^\phi|-|E^\phi|-1} \kappa(G_+) \kappa(G_-).
  \]
\end{corollary}
The reader will notice that the main theorem really pertains to graphs
with an involutive automorphism whose fixed edges have fixed
vertices. Indeed any such graph possess a drawing that makes it a
graph with reflective symmetry, as we have defined it: Draw the
quotient $G/\phi$ so that its fixed edges lay along a straight line
$\ell$. This drawing along with its reflection across $\ell$ yields
the desired drawing of $G$.

This paper is organized as follows. In Section~\ref{sec:critical} we
recall the definition of the critical group of a graph. In
Section~\ref{sec:map} we define the group homomorphism alluded to in
the Main Theorem. We then identify some basic properties of this
map. In Section~\ref{sec:ker-coker} we explicitly identify the kernel
and cokernel of the maps, and in Section~\ref{sec:relate} we relate
the order of these two objects. Finally, we include an appendix that
contains many technical facts on critical groups that are not
available in the published literature.

\section{Critical groups of graphs}\label{sec:critical}
For this section only, $G=(V,E)$ is an arbitrary graph. Orient the
edges of $G$ arbitrarily, and form the usual boundary map
\[
\partial = \partial(G) :\ZZ E \to \ZZ V.
\]
We follow the convention that the negative of an oriented edge $e \in
\ZZ E$ corresponds to that edge with opposite orientation. Thus, if
$e=uv$ is an oriented edge, then $uv = -vu$.  

Both $\ZZ E$ and $\ZZ V$ come with distinguished bases, and hence
orthonormal forms. There is, thus, an adjoint (or transpose) map,
\[
\partial^t : \ZZ V \to \ZZ E,
\]
given by the coboundary operator. The \textbf{bond (or cut) space} $B$ of
$G$ is the image of $\partial^t$. The \textbf{cycle space} $Z$ of $G$ is
the kernel of $\partial$. These spaces are free modules and are
orthogonal under the above form.

The bond space of $G$ is generated, as a free $\ZZ$-module, by the
fundamental bonds at the vertices of $G$, omitting one vertex from each
connected component. Given a vertex of $G$, this is the element
\[
\bb_G(v) := \sum_{u \sim v} uv \in \ZZ E.
\]
More generally, the fundamental bond of a subset $S \subset V$ is
$\bb_G(S) = \sum_{v \in S} \bb_G(v)$.

The cycle space of $G$ is generated, as free $\ZZ$-module, by oriented
circuits of $G$. That is, if $v_1 \to v_2 \to \dots \to v_\ell \to
v_{\ell+1}=v_1$ is an oriented circuit in $G$ then $\sum_{i=1}^\ell
v_iv_{i+1}$ is an element of the cycle space of $G$, and such elements
generate $Z$.

Following Appendix~\ref{sec:theory}, we define the \textbf{critical
  group of $G$} to be the quotient
\[
K(G):=(\ZZ E)/(Z+B).
\]
This is a finite abelian group whose order is the number of spanning
forests of $G$, which we denote by $\kappa(G)$.

\section{A map arising from a reflective symmetry}\label{sec:map}
Let $G=(V,E)$ be a graph with reflective symmetry, and
$G_+=(V_+,E_+)$, $G_-=(V_-,E_-)$ the left and right graphs this gives
rise to, as in the introduction. Let $\phi$ denote the involution
determined by the reflective symmetry.

The edges of $G_+$ come in two flavors: Those that are simply edges
from $E_L$, the left half $G$, and those that were obtained by
subdividing $\phi$-fixed edges of $G$. Edges of $G_-$ correspond
uniquely to edges in $E_R$, the right half of $G$. We will often abuse
notation and identify edges in $G_+$ or $G_-$ that come from edges in
$G$ with the corresponding edge in $G$. A similar identification will be
made for vertices in $G_+$ and $G_-$ arising from vertices in $G$.

Orient the edges of $G$ in such a way that $\phi$ is involution of
directed graphs. In this way we obtain induced orientations of $G_+$
and $G_-$.
\tikzstyle{vertex}=[draw, circle, minimum size=10pt, inner sep=1pt]
\tikzstyle{subvertex}=[draw, circle, fill=red!40,minimum size=10pt, inner sep=1pt]
\tikzstyle{edge} = [draw,thick,->,>=stealth',shorten >=1pt]
\begin{example}
  Here we display an orientation of $G$ that is $\phi$-fixed, and the
  induced orientations on $G_+$ and $G_-$.
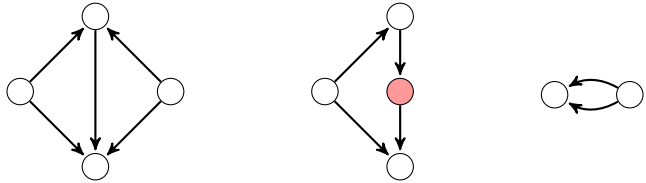
  $$
  \begin{tikzpicture}
    \foreach \pos/\name in { {(0,-1)/b}, {(0,1)/c},{(-1,0)/a}, {(1,0)/d}}
    \node (\name) at \pos [vertex] {};
    \foreach \source/ \dest in { a/b, a/c, d/c,d/b,c/b}
    \path[edge] (\source) --  (\dest);
  \end{tikzpicture}
  \qquad\qquad
  \begin{tikzpicture}
    \node (bc) at (0,0) [subvertex] {};
    \foreach \pos/\name in { {(0,-1)/b}, {(0,1)/c},
      {(-1,0)/a}}
    \node (\name) at \pos [vertex] {};
    \foreach \source/ \dest in { a/b, a/c, c/bc,bc/b}
    \path[edge] (\source) --  (\dest);
  \end{tikzpicture}
  \qquad\qquad
  \begin{tikzpicture}
    \node at (0,1) {}; \node at (0,-1) {};
    \foreach \pos/\name in { {(0,0)/bc},
      {(1,0)/d}}
    \node (\name) at \pos [vertex] {};
    \draw [draw,thick,->,>=stealth',shorten >=1pt] (d) to [bend right] (bc);
    \draw [draw,thick,->,>=stealth',shorten >=1pt] (d) to [bend left] (bc);
  \end{tikzpicture}
  $$  
\end{example}
Our primary goal in this section is to define a map
\[
f^*:K(G_+) \oplus K(G_-) \to K(G).
\]
For this we define a $\ZZ$-linear map
\[
f: \ZZ E_+ \oplus \ZZ E_- \to \ZZ E,
\]
that will take cycles to cycles and bonds to bonds, and $f^*$ will be
the induced map on critical groups. If $e \in E_+$ is an edge obtained
from a non-fixed edge of $E$, define
\[
f(e,0) := e + \phi(e).
\]
If $e \in E_+$ is obtained by subdividing an edge $e' \in E$, we set
$f(e,0) := e'$. For an edge $e \in E_-$, which we think of as an edge
in $E$, we set
\[
f(0,e) := e - \phi(e).
\]
There is the usual adjoint map $f^t : \ZZ E \to \ZZ E_+ \oplus \ZZ
E_-$, characterized by the property that $\<f(e',e''),e\> =
\<(e',e''),f^t(e)\>$. Specifically, for $e \in E^\phi$, let $e'$ and
$e''$ be the edges of $G_+$ this gives rise to. Then $f^t(e) =
(e'+e'',0)$. If $e \in E_L$ then $f^t(e) = (e,-\phi(e))$, and for $e
\in E_R$ we have $f^t(e) = (\phi(e),e)$.

Denote the cycle and bond spaces of $G_+$ by $Z_+$ and
$B_+$. Similarly denote the cycle and bond spaces of $G_-$ by $Z_-$
and $B_-$.
\begin{proposition}
  The map $f$ takes $Z_+ \oplus Z_-$ into $Z$, and takes $B_+ \oplus
  B_-$ into $B$.
\end{proposition}

\begin{example}
  We illustrate the proposition in our running example. Edges with
  arrows have coefficient $+1$ oriented in the indicated
  direction. Edges with larger coefficients are indicated. Edges
  without arrows have coefficient zero.

  Here we map a cycle $z$ in $G_+$ to a cycle in $G$ as $f(z,0)$.
  $$
  
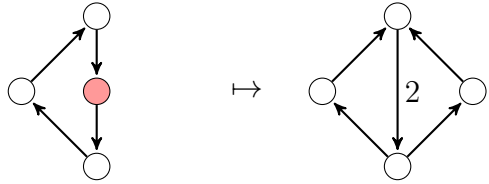
\begin{tikzpicture}
    \node (bc) at (0,0) [subvertex] {};
    \foreach \pos/\name in { {(0,-1)/b}, {(0,1)/c},
      {(-1,0)/a}}
    \node (\name) at \pos [vertex] {};
    \foreach \source/ \dest in { b/a, a/c, c/bc,bc/b}
    \path[edge] (\source) --  (\dest);
    \node at (2,0) {$\mapsto$};
    \node (aa) at (4,1) [vertex] {}; \node (cc) at (4,-1) [vertex] {};
    \node (bb) at (3,0) [vertex] {};\node (dd) at (5,0) [vertex] {};
    \draw [draw,thick,->,>=stealth',shorten >=1pt] (dd) to (aa);
    \draw [draw,thick,->,>=stealth',shorten >=1pt] (cc) to (dd);
    \draw [draw,thick,->,>=stealth',shorten >=1pt] (cc) to (bb);
    \draw [draw,thick,->,>=stealth',shorten >=1pt] (bb) to (aa);
    \draw [draw,thick,->,>=stealth',shorten >=1pt] (aa) to (cc);
    \node at (4.2,0) {$2$};
  \end{tikzpicture}
  $$
  Here we map a cycle $z$ in $G_-$ to a cycle in $G$ as $f(0,z)$.
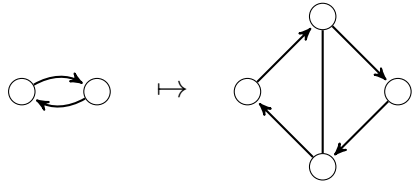
  $$
  \begin{tikzpicture}
    \node at (0,1) {}; \node at (0,-1) {};
    \foreach \pos/\name in { {(0,0)/bc},
      {(1,0)/d}}
    \node (\name) at \pos [vertex] {};
    \draw [draw,thick,->,>=stealth',shorten >=1pt] (bc) to [bend left] (d);
    \draw [draw,thick,->,>=stealth',shorten >=1pt] (d) to [bend left] (bc);
    \node at (2,0) {$\mapsto$};
    \node (aa) at (4,1) [vertex] {}; \node (cc) at (4,-1) [vertex] {};
    \node (bb) at (3,0) [vertex] {};\node (dd) at (5,0) [vertex] {};
    \draw [draw,thick,->,>=stealth',shorten >=1pt] (aa) to (dd);
    \draw [draw,thick,->,>=stealth',shorten >=1pt] (dd) to (cc);
    \draw [draw,thick,->,>=stealth',shorten >=1pt] (cc) to (bb);
    \draw [draw,thick,->,>=stealth',shorten >=1pt] (bb) to (aa);
    \draw [draw,thick,-] (aa) to (cc);
  \end{tikzpicture}
  $$
\end{example}

\begin{proof}
  We leave the proof of the statement about cycles to the reader,
  confident that the example will guide their proof.

  For the second part of the proposition it is sufficient to observe
  the following. First, if $v \in V_+$ is one of the vertices obtained
  by subdividing a $\phi$-fixed edge of $G$, then $f(\bb_{G_+}(v),0)
  = 0$. If $v \in V_+$ comes from a $\phi$-fixed vertex of $G$ then
  $f(\bb_{G_+}(v),0) = \bb_G(v)$. If $v \in V_+$ is any other vertex
  then $f(\bb_{G_+}(v),0) = \bb_G(v) + \bb_G(\phi(v))$.

  We now consider the case of $v \in V_-$. If $v$ is the vertex
  obtained by contracting $V^\phi \subset V$ to a point then
  $f(0,\bb_{G_-}(v)) = \bb_G(V_L)-\bb_G(V_R)$. For any other vertex,
  $f(0,\bb_{G_-}(v)) = \bb_G(v) - \bb_G(\phi(v))$.
\end{proof}

It follows that $f$ induces a natural map,
\[
f^* : K(G_+) \oplus K(G_-) \to K(G),
\]
on the quotient spaces. The adjoint $f^t$ induces a map going the
opposite direction
\[
(f^t)^* : K(G) \to K(G_+) \oplus K(G_-).
\]

\begin{proposition}
  The kernel and cokernel of $f^*$ are $2$-torsion. 
\end{proposition}
\begin{proof}
  By Proposition~\ref{prop:pontryagin} it is sufficient to prove that
  $f^*$ and $(f^t)^*$ have cokernels that are $2$-torsion.

  Choose an edge $e \in E$. If $e$ is $\phi$-fixed then $e \in
  \im(f)$. If $e$ is not $\phi$-fixed, suppose that $e$ is on the left
  half of $G$. We may view $e$ and $\phi(e)$ as edges in $G_+$ and
  $G_-$ and compute,
  \[
  f(e,-\phi(e)) = e+\phi(e)-\phi(e) + e = 2e.
  \]
  We conclude from this that $\coker(f)$ is $2$-torsion, and hence
  $\coker(f^*)$ is too.

  Choose an edge $e' \in E_+$. If $e'$ was obtained by subdividing $e
  \in E$, let $e'' \in E_+$ be the other edge obtained in this way. We
  compute,
  \[
  f^t(e) = e' + e'' \equiv e' + e'' + (e' - e'') = 2e' \mod B_+,
  \]
  since $e'-e''$ is a bond of $G_+$.

  If $e \in E_+$ did not arise from subdividing an edge of $G$, then
  we may identify $e$ with an edge $e$ of $G$. We have
  \[
  f^t( e-\phi(e) ) = f^t(e) - f^t(\phi(e)) =(e, \phi(e)) -
  (-e,\phi(e)) = 2(e,0).
  \]
  A similar computation shows that if $e \in E_-$ then $2(0,e) \in
  \im(f^t)$. It follows that $\coker((f^t)^*)$ is $2$-torsion.
\end{proof}
We have thus proved the first and second part of the Main Theorem.

\section{Identifying the kernel and cokernel}\label{sec:ker-coker}
To ease the notation within this section and the next we make the
following convention.
\begin{convention*}
  In this section and the next, $Z$ and $B$ will denote the reduction
  of the usual cycle and bond spaces of $G$ by the prime $2$. Thus $Z
  = \ker(\partial : (\ZZ/2)E \to (\ZZ/2)V)$ and $B = \im(\partial^t :
  (\ZZ/2)V \to (\ZZ/2)E)$. The same notation is used for $Z_{\pm}$ and
  $B_{\pm}$. We will also write $f$ and $f^t$ for the reduction of
  these $\ZZ$-linear maps by $2$.
\end{convention*}
Since the kernel and cokernel of $f^*$ are $2$-torsion their structure
is intimately related to the reduction of their critical groups by
$2$.  The $2$-bicycle space (hereafter the \textbf{bicycle space}) of
$G$ is $Z \cap B \subset (\ZZ/2)E$, which by
Proposition~\ref{prop:2bicycles} is naturally isomorphic to
$K(G)/2K(G)$.

An element $h$ of $(\ZZ/2)E$ can be identified with a subgraph of $H
\subset G$ via its support. Note that this subgraph does not come with
an orientation. An element of $Z \cap B$ corresponds to a graph $H$
satisfying the properties:
\begin{enumerate}
\item $H$ is the set of edges connecting a bipartition of $V$,
\item Every vertex of $G$ is incident to an even number of edges of
   $H$.
\end{enumerate}

The following algebraic result is proved in a more general context as
Proposition~\ref{prop:2bicycles}, and it follows since the kernel and
cokernel of $f^*$ are known to be $2$-torsion.
\begin{proposition}\label{prop:kercoker}
  There are group isomorphisms,
  \begin{align*}
  \coker(f^*) &\approx \ker( f^t : Z \cap B \to (Z_+ \oplus Z_-) \cap
  (B_+ \oplus B_-)),\\
  \ker(f^*) &\approx \ker( f : (Z_+ \oplus Z_-) \cap
  (B_+ \oplus B_-) \to Z \cap B).
  \end{align*}
\end{proposition} 

We are now in a position to identify $\coker(f^*)$.
\begin{proposition}\label{prop:coker}
  The kernel of $f^t:(\ZZ/2) E \to (\ZZ/2)E_+ \oplus (\ZZ/2)E_-$ has a
  basis given by the $\phi$-fixed elements $e + \phi(e)$. The kernel
  of $f^t$ restricted to $Z \cap B$ consists of the $\phi$-fixed
  bicycles of $G$.
\end{proposition}
\begin{proof}
  It is sufficient to prove the first claim. A basis for $(\ZZ/2)E$ is
  given by $\{ e + \phi(e) : e \in E_L\} \cup E_L \cup
  E^\phi$. Likewise, a basis of $(\ZZ/2)E_+ \oplus (\ZZ/2)E_-$ is
  given by $\{(e,\phi(e)): e \in E_L\} \cup E_+ \cup E_L$.

  The matrix of $f^t$ becomes diagonal in this basis, and it is clear
  that the kernel of $f^t$ has the stated form.
\end{proof}

To identify the cokernel of $(f^t)^*$ we need another
involution. Define $$\psi:(\ZZ/2)(E_+ \cup E_-) \to (\ZZ/2)(E_+ \cup
E_-)$$ as follows. If $e \in E_+$ is obtained by subdividing an edge
of $E$ then set $\psi(e)$ equal to the other edge obtained in this
way. If $e \in E_+$ arises from an edge that is not $\phi$-fixed, then
we define $\psi(e) := \phi(e) \in E_-$ and $\psi(\phi(e)) = e$. This
map is \textit{not} determined by a graph automorphism.
\begin{proposition}\label{prop:ker}
  The kernel of $f: (\ZZ/2)E_+ \oplus (\ZZ/2)E_- \to (\ZZ/2)E$
  consists of the elements $\psi$-fixed elements.  The kernel of $f$
  restricted to $(Z_+ \cap Z_-) \cap (B_+ \cap B_-)$ consists of the
  $\psi$-fixed bicycles of $G_+ \cup G_-$.
\end{proposition}
\begin{proof}
  Compute the matrix of $f$ in term of the basis used in the proof of
  Proposition~\ref{prop:coker}. The first statement follows from
  inspection of the matrix representing $f$ and the second follows
  from the first.
\end{proof}

Propositions \ref{prop:kercoker}, \ref{prop:coker} and \ref{prop:ker}
give a complete combinatorial description of $\coker(f^*)$ and
$\ker(f^*)$. We illustrate them with an example.

\begin{example}
  We continue with our running example, starting with $\coker(f^*)$. A
  $\phi$-fixed bicycle in $G$ is indicated by the shaded edges below.
  \tikzstyle{edge} = [draw,thick,-] \tikzstyle{selected edge} =
  [draw,line width=5pt,-,red!50]
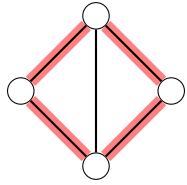
  $$
  \begin{tikzpicture}
    \foreach \pos/\name in { {(0,-1)/b}, {(0,1)/c},{(-1,0)/a}, {(1,0)/d}}
    \node (\name) at \pos [vertex] {};
    \path[selected edge] (a) -- (b);
    \path[selected edge] (a) -- (c);
    \path[selected edge] (d) -- (c);
    \path[selected edge] (d) -- (b);
    \foreach \source/ \dest in { a/b, a/c, d/c,d/b,c/b}
    \path[edge] (\source) --  (\dest);
  \end{tikzpicture}
  $$
  It follows that $\coker(f^*)\approx \ZZ/2$. For the kernel of $f^*$,
  we investigate bicycles in $G_+ \cup G_-$. The graph $G_+ \cup G_-$
  itself is a $\psi$-fixed bicycle.
  $$
    
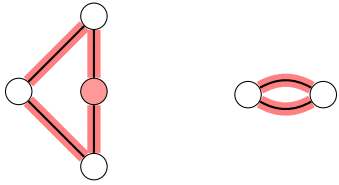
\begin{tikzpicture}
    \node (bc) at (0,0) [subvertex] {};
    \foreach \pos/\name in { {(0,-1)/b}, {(0,1)/c},
      {(-1,0)/a}}
    \node (\name) at \pos [vertex] {};
    \foreach \source/ \dest in { a/b, a/c, c/bc,bc/b}
    \path[selected edge] (\source) --  (\dest);
    \foreach \source/ \dest in { a/b, a/c, c/bc,bc/b}
    \path[edge] (\source) --  (\dest);
  \end{tikzpicture}
  \qquad\qquad
  \begin{tikzpicture}
    \node at (0,1) {}; \node at (0,-1) {};
    \foreach \pos/\name in { {(0,0)/bc},
      {(1,0)/d}}
    \node (\name) at \pos [vertex] {};
    \draw [selected edge] (d) to [bend right] (bc);
    \draw [selected edge] (d) to [bend left] (bc);
    \draw [edge] (d) to [bend right] (bc);
    \draw [edge] (d) to [bend left] (bc);
  \end{tikzpicture}
  $$
  Although both $G_+ \subset G_+ \cup G_-$ and $G_- \subset G_+ \cup
  G_-$ are bicycles, they are not $\psi$-fixed. It follows that
  $\ker(f^*)\approx \ZZ/2$. The $\ker$-$\coker$ exact sequence for
  $f^*$ takes the form,
  \[
  0 \to \ZZ/2 \to \ZZ/4 \oplus \ZZ/2 \stackrel{f^*}{\to} \ZZ/8 \to \ZZ/2
  \to 0.
  \]
\end{example}
\begin{example}
  Let $G$ be a $(2n)$-cycle with the obvious reflective symmetry. Then
  $G_+$ is a path on $n+1$ vertices and $G_-$ is an $n$-cycle. We see
  that $G$ is a $\phi$-fixed bicycle. Since $G_+$ is a path it has no
  (non-empty) bicycles, and hence there are no $\psi$-fixed bicycles.

  The $\ker$-$\coker$ exact sequence for the map $f^*$ takes the form
  \[
  0 \to \ZZ/n \to \ZZ/(2n) \to \ZZ/2 \to 0,
  \]
  which is never split if $n$ is even.
\end{example}
We close this section with alternate presentations of $\ker(f^*)$ and
$\coker(f^*)$.

\begin{proposition}\label{prop:alternateKerCoker}
  There are isomorphisms,
  \begin{align*}
    \ker(f^*).&\approx \frac{((\ZZ/2)(E_+ \cup E_-))^\psi}{(Z_+ \oplus Z_-)^\psi + (B_+
    \oplus B_-)^\psi}, \\
  \coker(f^*) &\approx \frac{((\ZZ/2)E)^\phi}{Z^\phi + B^\phi}.
  \end{align*}
\end{proposition}
\begin{proof}
  The proofs amount to the fact that $\ker(f^*)$ and $\coker(f^*)$ are
  succinctly described as the $\psi$ and $\phi$ fixed elements of
  $K(G_+)/2K(G_+) \oplus K(G_-)/2K(G_-)$ and $K(G)/2K(G)$. This is
  true because the isomorphisms relating the various presentations of
  the critical groups in Appendix~\ref{sec:bicycles} are equivariant
  with respect to $\psi$ and $\phi$.

  We then see that $((\ZZ/2)(E_+ \cup E_-))^\psi$ and
  $((\ZZ/2)E)^\phi$ surject onto these critical groups. The kernels of
  these maps are evident.
\end{proof}

%%%%%%%%%%%%%%%%%%%%%%%%%%%%%%%%%%%%%%%%%%%%%%%%%%%%%%%%%%%%
%%%%%%%%%%%%%%%%%%%%%%%%%%%%%%%%%%%%%%%%%%%%%%%%%%%%%%%%%%%%
%%%%%%%%%%%%%%%%%%%%
%%%%%%%%%%%%%%%%%%%%%%%%%%%%%%%%%%%%%%%%%%%%%%%%%%%%%%%%%%%%
%%%%%%%%%%%%%%%%%%%%%%%%%%%%%%%%%%%%%%%%%%%%%%%%%%%%%%%%%%%%

\section{Relating $\ker(f^*)$ and $\coker(f^*)$}\label{sec:relate}
Our final goal is to relate the orders of $\ker(f^*)$ and
$\coker(f^*)$ in a concrete fashion. An easy and immediate result is
that $|\coker(f^*)| / |\ker(f^*)|$ is a positive integer power of $2$.
\begin{proposition}\label{prop:kerIntoCoker}
  There is an injective map
  \[
  \ker(f^*)  \to \coker(f^*).
  \]
\end{proposition}
\begin{proof}
  We use the above presentation of these groups as $\psi$ and $\phi$
  fixed bicycles. If $(x,x') \in (\ZZ/2)E_+ \oplus (\ZZ/2)E_-$ is
  $\psi$-fixed, set $g(x,x'):=f(x,0) = f(0,x')$. This restricts to a
  map on the $\psi$-fixed bicycles whose image is in the space of
  $\phi$-fixed bicycles. The map is injective since $f|_{(Z_+ \oplus
    Z_-)^\psi}$ is injective.
\end{proof}

We will use the map $g$ occurring in the proof of the proposition in
what follows. Consider the commutative diagram below, whose horizontal
arrows are those induced by $g$, and whose vertical arrows are the
natural ones.
\begin{align}\label{eq:snake}
\xymatrix{ 0 \ar[d]& 0\ar[d]\\ (Z_+ \oplus Z_-)^\psi \cap (B_+ \oplus B_-)^\psi \ar[r]\ar[d] &
  Z^\phi \cap B^\phi \ar[d]\\
 (Z_+ \oplus Z_-)^\psi \oplus (B_+ \oplus B_-)^\psi \ar[r]\ar[d] &
  Z^\phi \oplus B^\phi\ar[d] \\
 (Z_+ \oplus Z_-)^\psi + (B_+ \oplus B_-)^\psi \ar[r] \ar[d]&
  Z^\phi + B^\phi\ar[d] \\  0 & 0}  
\end{align}
The columns in this diagram are exact. We wish to identify the order
of the cokernel in the top row. For this, we need to compute the
kernel and cokernel in the middle row.
\begin{proposition}\label{prop:forSnake1}
  The dimension of $(B_+ \oplus B_-)^\psi$ is $|V_R| + |E^\phi|$. The
  dimension of $B^\phi$ is $|V_R| + |V^\phi| - 1$. It follows that
  \[
  |B^\phi|/|(B_+ \oplus B_-)^\psi|= 2^{|V^\phi| - |E^\phi| - 1}
  \]
\end{proposition}
\begin{proof}
  A basis for the bond space of $G_-$ is obtained by taking the
  fundamental bonds at all of its vertices except one. We exclude the
  vertex obtained by contracting all of $V^\phi$ to a point. If we
  take these bonds and symmetrize them by $\psi$ we obtain
  $|V_R|=|V_L|$ many linearly independent bonds in $(B_+ \oplus
  B_-)^\psi$. Any $\psi$-fixed bond not contained in the span of these
  cannot be supported on $B_-$. It is clear that the bonds at the
  vertices obtained by subdividing $\phi$-fixed edges complete our
  description of a basis of $(B_+ \oplus B_-)^\psi$.

  A basis for the bond space of $G$ is given by all but one of the
  fundamental bonds at vertices of $G$. We omit a $\phi$-fixed vertex
  from our basis. The remaining $\phi$-fixed vertices have
  $\phi$-fixed bonds. Symmetrizing the bonds of vertices in $V_R$
  yields the rest of a basis for $B^\phi$.
\end{proof}

\begin{proposition}\label{prop:forSnake2}
  Suppose that $G_+$ is connected. There is an equality,
  \[
  \frac{|(Z_+ \oplus Z_-)^\psi|}{|Z^\phi|} = 2^{|V^\phi| - |E^\phi|
    -1}.
  \]
\end{proposition}
\begin{proof}
  The idea is to consider the injection $g:(Z_+ \oplus Z_-)^\psi \to
  Z^\phi$, and compute a basis for its cokernel. For this we note that
  $|V^\phi| - |E^\phi| -1$ is the number of connected components of
  $G^\phi$.

  Choose one vertex from each connected component of $G^\phi$,
  $v_0,v_1,\dots, v_m$. In $G_+$, take a path $p_{ij}$ connecting
  $v_i$ to $v_j$. Viewing $p_{ij}$ as a path in $G$, we form the cycle
  $z_{ij} = p_{ij} + \phi(p_{ij}) \in Z^\phi$. We claim that the
  cycles $\{z_{ij}\}$ are not in the image of $g$. If there was a
  cycle in $G_-$ lifting $z_{ij}$ then it would differ from $p_{ij}$
  by a sum of bonds of vertices obtained by subdividing $\phi$-fixed
  edges. Since $v_i$ and $v_j$ are not connected by a path in $G^\phi$
  we see that $z_{ij} \notin \im(g)$.

  Let $v_i'$ and $v_j'$ be two vertices in the same connected
  component of $G^\phi$ as $v_i$ and $v_j$, respectively. If $p'_{ij}$
  is a path connecting $v'_i$ to $v'_j$ and $z'_{ij} = p'_{ij} +
  \phi(p'_{ij})$, then $z_{ij} + z'_{ij} \in \im(g)$. This is because
  we have a cycle of $G_+$, $p_{ij} + p'_{ij} +$(a subdivided path in
  $G^\phi$ from $v_i$ to $v'_i$)+(a subdivided path in $G^\phi$ from
  $v_j$ to $v'_j$). Applying $\psi$ to this cycle yields a cycle in
  $G_-$, since $p_{ij}$ must touch the axis of symmetry an even number
  of times. We conclude from this $z_{ij} = z'_{ij} \in Z^\phi/\im(g)$
  and that $z_{i(i+1)} + \dots +z_{(j-1)j}$ is equivalent to $z_{i,j}$
  in $Z^\phi/\im(g)$.

  We now claim that (the images of) $z_{01}, z_{12},\dots,z_{(m-1)m}$
  form a basis for $Z^\phi / \im(g)$. They are linearly independent
  since the paths $\{p_{i(i+1)}\}$ are linearly independent in
  $(\ZZ/2) E_+$. If $z$ represents a cycle in $G$ that is not in
  $\im(g)$, then $z$ visits each of some even number of connected
  component of $G^\phi$ an twice odd number of times. Subtracting off
  elements of the form $z_{ij}$, where $i$ and $j$ label two
  components visited an odd number of times by $z$, we conclude that
  $\{z_{ij}\}$ spans $Z^\phi/\im(g)$. From this we have $\{z_{i(i+1)}:
  i=0,\dots,m-1\}$ spans $Z^\phi/\im(g)$
\end{proof}

\begin{proposition}\label{prop:forSnake3}
  There is an equality,
  \[
  \frac{|(Z_+ \oplus Z_-)^\psi + (B_+ \oplus
    B_-)^\psi|}{|Z^\phi + B^\phi|} = \frac{|\ker(f^*)|}{|\coker(f^*)|}
  \]
\end{proposition}
\begin{proof}
  This follows from Proposition~\ref{prop:alternateKerCoker} by
  multiplying and dividing the left side by $|((\ZZ/2)E)^\phi|$, which
  is equal to $|((\ZZ/2)(E_+ \cup E_-))^\psi|$.
\end{proof}

We are \textit{finally} in a position to prove the remaining part of
the Main Theorem.
\begin{theorem}
  Suppose that $G_+$ is connected. Then,
  \[
  \frac{|K(G_+) \oplus K(G_-)|}{|K(G)|}
  =
  \frac{|\ker(f^*)|}{|\coker(f^*)|} = 2^{|V^\phi|-|E^\phi|-1}.
  \]
\end{theorem}
\begin{proof}
  Take the alternating product of the orders of the groups in the
  first column of the diagram \eqref{eq:snake} and divide by the
  alternating product for the second column. We obtain,
  \[
  1=\frac{|(Z_+ \oplus Z_-)^\psi \cap (B_+ \oplus B_-)^\psi|\cdot
    |(Z_+ \oplus Z_-)^\psi + (B_+ \oplus B_-)^\psi| \cdot |Z^\phi
    \oplus B^\phi|}{|(Z_+ \oplus Z_-)^\psi \oplus (B_+ \oplus
    B_-)^\psi|\cdot |Z^\phi \cap B^\phi| \cdot |Z^\phi +
    B^\phi|}.
  \]
  Applying Propositions~\ref{prop:kerIntoCoker}, \ref{prop:forSnake1},
  \ref{prop:forSnake2}, and \ref{prop:forSnake3} this yields,
  \[
  1 = 2^{2(|V^\phi|-|E^\phi|-1)} \frac{|\ker(f^*)|^2 }{|\coker(f^*)|^2}.
  \]
  Manipulating this fraction and taking the square root proves the
  theorem.
\end{proof}
\section{Open problems}
It would be desirable to actually exhibit bicycles forming a basis of
the cokernel of $g: \ker(f^*) \to \coker(f^*)$. This appears to be
difficult and subtle, since it requires producing linearly independent
bicycles in $G$. When $G$ is planar the left-right tours of Shank
\cite{shank} could possibly be used to furnish the needed bicycles.

There is a more general version of Theorem~\ref{thm:motivation} given
by Yan and Zhang \cite{yanZhang}. It allows for an arbitrary
involution on a weighted graph drawn in the plane, essentially meaning
that we relax the condition that edges cannot cross the axis of
symmetry.

The construction of an appropriate version of $G_+$ and $G_-$ is more
involved in this case, but a result of the form $\kappa(G) = 2^m
\kappa(G_+) \kappa(G_-)$ is obtained. The integer $m$ appearing in
this formula might be negative in general, and positive integer
weights for $G$ might involve half integer weights for $G_+$ and
$G_-$. It would be interesting to see a critical group generalization
of this result.

\appendix

\section{Critical groups of adjoint pairs}\label{sec:theory}
The point of these appendices is to gather and prove algebraic results
about critical groups for the previous work. Some of these results can
be found in Bacher, de la Harpe, Nagnebeda \cite{bacherEtAlia} and
Treumann's bachelors thesis \cite{treumann}

Following Treumann \cite{treumann}, we consider the category $\Adj$,
whose objects are adjoint pairs $(\partial, \partial^t)$ of linear
maps,
\[
\partial : C_1 \to C_0, \quad \partial^t : C_0 \to C_1
\]
between two finitely generated free $\ZZ$-modules $C_1$ and $C_0$. We
assume that both $C_1$ and $C_0$ are both equipped with a positive
definite inner product (both denoted $\<-,-\>$) and have bases which
are orthonormal with respect to these inner products. The
adjointness of the maps $\partial$, $\partial^t$ means that for all
$v \in C_0$ and $e \in C_1$,
\[
\<\partial e, v \> =\<e,\partial^t v\>.
\]

Let $(\partial,\partial^t)$ be an adjoint pair as above. We define $Z
:= \ker(\partial)$ and $B := \im(\partial^t)$. The \textbf{critical
  group} of $(\partial,\partial^t)$ is
\[
K=K(\partial,\partial^t) := C_1/(Z+ B).
\]

A morphism between two adjoint pairs $(\partial,\partial^t)$,
$(\partial',(\partial')^t)$ is a pair of linear maps
\[
f = (f_1 : C_1 \to C_1', f_0 : C_0 \to C_0' ),
\]
subject to the intertwining conditions
\[ 
f_0 \partial = \partial' f_1, \qquad f_1 \partial^t = (\partial')^t
f_0 \mod B'.
\]

A morphism $f=(f_1,f_0)$ between two pairs $(\partial, \partial^t)$
and $(\partial',(\partial')^t)$ induces maps
\[
f^*_1 : K \to K',
\]
as one checks that $f_1$ takes $Z$ into $Z'$ and $B$ into $B'$. 
\begin{proof}
  Suppose that $z \in Z$. Then $\partial' f_1(z) = f_0 (\partial z) =
  0$. Suppose that $b = (\partial')^t v$. Then $f_1((\partial')^t v)
  =\partial^t f_0(v) \mod B'$, hence $f_1((\partial')^t v)$ is an
  element of $B'$.
\end{proof}
In this way, the critical group is a functor $\Adj \to \mathbf{Ab}$
from $\Adj$ to the category of finitely generated abelian groups.

There is an alternate definition of the critical group in terms of the
Laplacian operator $\partial \partial^t : C_0 \to C_0$.
\begin{proposition}[Treumann~{\cite{treumann}}]
  The induced map $\partial : K \to \coker( \partial \partial^t)$ is
  injective and there is a direct sum decomposition,
  \[
  K \oplus \coker(\partial) = \coker(\partial \partial^t).
  \]
  The order of $K$ is the absolute value of the maximal minors of
  $\partial \partial^t$.
\end{proposition}
When $\partial = \partial(G)$ for a graph $G$, the cokernel of
$\partial$ is a free $\ZZ$-module whose rank is the number of
connected components of $G$. In this case the absolute value of the
maximal minors of $\partial \partial^t$ is the given by Kirchhoff's
Theorem as the number of spanning forests of $G$.

A morphism $f=(f_1,f_0)$ between two pairs $(\partial, \partial^t)$
and $(\partial',(\partial')^t)$ gives rise to a natural map on the
quotient
\[
f_0^* : \coker( \partial \partial^t) \to \coker( \partial'
(\partial')^t)
\]
This map restricts to the critical group summands, and we have the
following result.
\begin{theorem}[Treumann~{\cite{treumann}}]
  The induced maps
  \[
  f_1^* : K \to K', \qquad f_0^*: K \to K'
  \]
  are equal.
\end{theorem}
As such, we will drop the subscript on $f$ when referring to the map it
induces on critical groups.

Let $f$ be a morphism in $\Adj$, taking $(\partial,\partial^t)$ to
$(\partial',(\partial')^t)$. The map $f$ has an adjoint $f^t$ taking
$(\partial',(\partial')^t)$ to $(\partial, \partial^t)$, given by
taking the adjoint of the constituent functions. This means that
\[
\< f_1 e, e' \> = \< e, f_1^t e' \>, \qquad \< f_0 v, v' \> = \< v,
f_0^t v' \>.
\]
The pair $f^t = (f_1^t,f_0^t)$ is a morphism in $\Adj$. The maps $f^*
: K \to K'$ and $(f^t)^*$ are related by the following result.
\begin{proposition}\label{prop:pontryagin}
  There is a commutative square
  \[
  \xymatrix{
    K  \ar[d]^{\approx}  & \ar[l]^{{f^t}^*} K' \ar[d]_{\approx}\\
    \Hom_\ZZ(K,\QQ/\ZZ)\ar@{<-}[r]^{\circ f^*} & \Hom_\ZZ(K',\QQ/\ZZ)
  }
  \]
  There are natural isomorphisms $\ker(f^*) \approx \coker({f^t}^*)$,
  $\coker(f^*) \approx \ker({f^t}^*)$.
\end{proposition}
It is a fact that $K$ is equipped with a non-degenerate
$\QQ/\ZZ$-valued bilinear form $(-,-)$, see
\cite[p.170]{bacherEtAlia}, \cite[p.3]{treumann}.  The identification
of $K$ with $\Hom_\ZZ(K,\QQ/\ZZ)$ is via $x \mapsto (x,-)$.  For this
inner product we have the relation,
\[
( f^*(x), y) = (x, {f^t}^*(y) ).
\]
See \cite[Proposition 2.5]{lineCrit} or \cite[Proposition~9]{treumann}
for the proof of the proposition.

We now come to an important technical lemma.
\begin{lemma}\label{lem:f0}
  Let $(\partial,\partial^t)$ and $(\partial',(\partial')^t)$ be two
  objects in $\Adj$. Suppose that $f_1 :C_1 \to C_1'$ is a
  $\ZZ$-linear map satisfying $f_1 Z \subset Z'$ and $f_1 B \subset
  B'$.

  Define $f_0: \im \partial \to C_0'$ by $f_0 (\partial x)
  := \partial' f_1(x)$. Then $f:=(f_1,f_0)$ defines a morphism in $\Adj$,
  \[
  ( \partial: C_1 \to \im \partial, \partial^t : \im \partial \to
  C_1) \stackrel{f}{\to} (\partial',(\partial')^t).
  \]
\end{lemma}
\begin{proof}
  This is well defined since $f_1$ takes $Z$ into $Z'$. We only need
  to check that the second intertwining condition is satisfied, since
  the first is satisfied by definition. For this we must have that
  $f_1$ commutes, up to $B'$, with the down-up Laplacian:
  \[
  f_1 ((\partial^t \partial) x) = ((\partial')^t \partial') f_1(x) \mod B'.
  \]
  On the left, since $(\partial^t \partial) x$ is in $B$ and $f_1$
  takes $B$ to $B'$ we obtain an element of $B'$. The element of the
  right is patently in $B'$, thus both sides are equal modulo $B'$.
\end{proof}
The point of this result is that if one has defined $f_1$ and it
behaves sufficiently well then $f_0$ is essentially determined by
$f_1$, modulo data that the critical group cannot see. Indeed, if
$(\partial,\partial^t)$ is in $\Adj$ then its critical group is equal
to that of
\[
( \partial: C_1 \to \im \partial, \partial^t : \im \partial \to C_1) \in \Adj.
\]
This follows since the image of $\partial^t: \im \partial \to C_1$ is
equal to the image of $\partial^t: C_0 \to C_1$.
\section{Bicycles}\label{sec:bicycles}
We maintain the notation of adjoint pairs from the previous
section. Let $p$ be a prime number. The \textbf{$p$-bicycles} of an
adjoint pair $(\partial,\partial^t)$ are the elements of
\[
\Hom(K,\ZZ/p).
\]
Since $K$ is equipped with a non-degenerate $\QQ/\ZZ$-valued bilinear
form the $p$-bicycles are naturally identified with $K/pK$. Indeed,
every map $K \to \ZZ/p$ can be thought of as a map $K \to
\{0,1/p,\dots,(p-1)/p\} \subset \QQ/\ZZ$, and this map is of the form
$(x,-)$ for some $x \in K$. Mapping $(x,-)$ to $x + pK$ gives the
desired identification.

A morphism $f:(\partial,\partial^t) \to (\partial',(\partial')^t)$ in
$\Adj$ gives rise to a natural map
\[
f: K/pK \to K'/pK'
\]
which is just reduction of $f$ by $p$. Given an adjoint pair
$(\partial,\partial^t)$ we will denote the reduction by $p$ of the
associated modules $Z$ and $B$ by $Z^{\ZZ/p}$ and $B^{\ZZ/p}$.

\begin{proposition}
  There is a commutative square
  \[
  \xymatrix{
    K/pK  \ar[d]^{\approx} & \ar[l]^{{f^t}^*} K'/pK'  \ar[d]_{\approx}\\
    \Hom(K,\ZZ/p) \ar@{<-}[r]^{\circ f^*} & \Hom(K',\ZZ/p) }
  \]
  If $\coker(f^*)$ is all $p$-torsion, there is a natural isomorphism
  $\coker(f^*) \approx \ker({f^t}^*)$.
\end{proposition}
\begin{proof}
  Let $(-,-)$ denote the $\QQ/\ZZ$-valued bilinear forms on $K$ and
  $K'$. The commutativity of the diagram boils down to the equality
  $(f(x),y) = (x,f^t(y))$, which holds by
  Proposition~\ref{prop:pontryagin}, reduced modulo
  $\frac{1}{p}(\QQ/\ZZ) \subset \QQ/\ZZ$.

  Since $\Hom_\ZZ(-,\ZZ/p)$ is right exact, the kernel of
  $(f^t)^*:K'/pK' \to K/pK$ is identified with
  $\Hom_\ZZ(\coker(f^*),\ZZ/p)$. Since $\coker(f^*)$ is assumed to be
  $p$-torsion, $\Hom_\ZZ(\coker(f^*),\ZZ/p) = \coker(f^*)$.
\end{proof}

\begin{proposition}\label{prop:2bicycles}
  The $p$-bicycles of $(\partial,\partial^t)$ are naturally identified
  with $Z^{\ZZ/p} \cap B^{\ZZ/p}$. This association is functorial in
  the sense that if $f: (\partial, \partial^t) \to (\partial',
  (\partial')^t)$ is a morphism, then there is a commutative diagram,
  \[
  \xymatrix{
    K/pK \ar[r]^{f^*}\ar[d]_{\approx} & K'/p K'\ar[d]^{\approx} \\
    Z^{\ZZ/p} \cap B^{\ZZ/p} \ar[r]_{f_1} & (Z')^{\ZZ/p} \cap (B')^{\ZZ/p}\\
  }
  \]
\end{proposition}
\begin{proof}
  There is a direct sum decomposition
  \[
  C_0/pC_0 = \ker( \partial \partial^t ) \oplus
  \im( \partial \partial^t ),
  \]
  since $\partial \partial^t$ is self-adjoint. This yields an
  isomorphism of $\ZZ/p$-vector spaces
  \[
  \ker( \partial \partial^t ) \approx \coker(\partial \partial^t) = K/pK
  \oplus \coker(\partial).
  \]
  Now, map an element $x$ in $\ker \partial \partial^t$ to $\partial^t
  x$. The kernel of this map is precisely $\coker(\partial)$, and its
  image is all of $Z^{\ZZ/p} \cap B^{\ZZ/p}$. This proves the first part of the
  result.

  The commutativity of the square follows directly from the relations 
  \[
  f_0 \partial = \partial' f_1, \qquad f_1 \partial^t = (\partial')^t
  f_0,
  \]
  and the fact that $f^*$ is represented either by $f_0^*$ or $f_1^*$.
\end{proof}

By Lemma~\ref{lem:f0}, we do not need to have the full data of a
morphism in $\Adj$ to reach the conclusion of the previous
proposition.
\begin{corollary}
  Let $(\partial,\partial^t)$ and $(\partial',(\partial')^t)$ be two
  objects in $\Adj$. Suppose that $f_1 :C_1 \to C_1'$ is a
  $\ZZ$-linear map satisfying $f_1 Z \subset Z'$ and $f_1 B \subset
  B'$. Then, there is a commutative diagram
  \[
  \xymatrix{
    K/pK \ar[r]^{f^*}\ar[d]_{\approx} & K'/p K'\ar[d]^{\approx} \\
    Z^{\ZZ/p} \cap B^{\ZZ/p} \ar[r]_{f_1} & (Z')^{\ZZ/p} \cap (B')^{\ZZ/p}
  }
  \]
\end{corollary}

\section*{Acknowledgments} Thanks are due to an anonymous referee for
a careful reading of the manuscript.

\bibliography{inv}{}
\bibliographystyle{plain}

\end{document}